\documentclass[a4paper,12pt,final]{amsart}
\usepackage{times,a4wide,mathrsfs,amssymb,dsfont}

\newcommand{\C}{\mathbb{C}}

\newcommand{\QQ}{\mathbb{Q}}
\newcommand{\NN}{\mathbb{N}}
\newcommand{\PP}{\mathbb{P}}

\newcommand{\OO}{\mathcal O}
\newcommand{\Ss}{\mathcal S}

\newcommand{\NNN}{\mathcal N}

\newcommand{\MM}{\mathcal M}

\newcommand{\UU}{\mathcal U}

\newcommand{\wt}{\widetilde}
\newcommand{\ima}{\hbox{Im}}

\newtheorem{theorem}{Theorem}[section]
\newtheorem{claim}[theorem]{Claim}

\newtheorem{corollary}[theorem]{Corollary}
\newtheorem{proposition}[theorem]{Proposition}

\newtheorem{nonumbering}{Theorem}

\newtheorem{nonumberingc}{Corollary}

\newtheorem{convention}{Conventions}

\theoremstyle{definition}
\newtheorem{remark}[theorem]{Remark}
\newtheorem{definition}[theorem]{Definition}

\newtheorem{nonumberingt}{Acknowledgements}

\begin{document}
\author[Robert Laterveer]
{Robert Laterveer}

\address{Institut de Recherche Math\'ematique Avanc\'ee,
CNRS -- Universit\'e 
de Strasbourg,\
7 Rue Ren\'e Des\-car\-tes, 67084 Strasbourg CEDEX,
FRANCE.}
\email{robert.laterveer@math.unistra.fr}

\title{A remark on the Chow ring of K\"uchle fourfolds of type $d3$}

\begin{abstract} We prove that K\"uchle fourfolds $X$ of type d3 have a multiplicative Chow--K\"unneth decomposition. We present some consequences for the Chow ring of $X$.
\end{abstract}



\maketitle

\section{Introduction}

K\"uchle \cite{Ku} has classified Fano fourfolds that are obtained as sections of globally generated homogeneous vector bundles on Grassmannians. In K\"uchle's list, a Fano fourfold 
{\em of type d3\/} is defined
as a smooth $4$--dimensional section of the vector bundle
  \[  (\wedge^2 \UU^\vee)^{\oplus 2}\oplus \OO_{G}(1)\ \to\ G:=\operatorname{Gr}(5,V_{10})\ ,\]
  where $\UU$ denotes the tautological bundle on the Grassmannian $G$ (of $5$--dimensional subspaces of a fixed $10$--dimensional complex vector space).

  Let $X$ be a K\"uchle fourfold of type d3.
  The Hodge diamond of $X$ is
    \[ \begin{array}[c]{ccccccccccccc}
      &&&&&& 1 &&&&&&\\
      &&&&&0&&0&&&&&\\
      &&&&0&&5&&0&&&&\\
        &&&0&&0&&0&&0&&&\\
      &&0&&1&&26&&1&&0&&\\
      &&&0&&0&&0&&0&&&\\
       &&&&0&&5&&0&&&&\\
        &&&&&0&&0&&&&&\\      
        &&&&&& 1 &&&&&&\\
\end{array}\]
   
 This Hodge diamond looks like the one of a $K3$ surface. Interestingly, the relation
   \[    \hbox{K\"uchle\ fourfolds\ of\ type\ d3}\ \leftrightsquigarrow    \ \hbox{$K3$\ surfaces} \]  
   goes further than a similarity of Hodge numbers: Kuznetsov \cite[Corollary 3.6]{Kuz} has shown that K\"uchle fourfolds $X$ of type d3 are related to $K3$ surfaces on the level of derived categories
  (in the sense that the derived category of $X$ admits a semi--orthogonal decomposition, and the interesting part of this decomposition is isomorphic to the derived category of a $K3$ surface).
 
 The main result of the present note is that K\"uchle fourfolds of type d3 behave like $K3$ surfaces from a Chow--theoretic point of view:
 
 \begin{nonumbering}[=theorem \ref{main}] Let $X$ be a K\"uchle fourfold of type d3. Then $X$ has a multiplicative Chow--K\"unneth decomposition (in the sense of \cite{SV}).
 \end{nonumbering}
 
 This is very easily proven, provided one uses Kuznetsov's alternative description \cite{Kuz} of K\"uchle fourfolds of type d3.
 Theorem \ref{main} has interesting consequences for the Chow ring $A^\ast(X)_{\QQ}$:
  
 \begin{nonumberingc}[=corollary \ref{cor}] Let $X$ be a K\"uchle fourfold of type d3. 
 Let $R^3(X)\subset A^3(X)_{\QQ}$ be the subgroup generated by the Chern class $c_3(T_X)$ and intersections $A^1(X)\cdot A^2(X)$ of divisors with $2$--cycles.
 The cycle class map induces an injection
   \[ R^3(X)\ \hookrightarrow\ H^6(X,\QQ)\ .\]
  \end{nonumberingc}
 
  This is reminiscent of the famous result of Beauville--Voisin describing the Chow ring of a $K3$ surface \cite{BV}.
  More generally, there is a similar injectivity result for the Chow ring of certain self--products $X^m$ (corollary \ref{cor}).
  
 There are two other families of fourfolds on K\"uchle's list which have a Hodge diamond of $K3$ type: the families of type c5 and c7. There is also the complete intersection eightfold in $\operatorname{Gr}(2,V_8)$ of \cite[Proposition 5.2]{FM} which has a Hodge diamond of $K3$ type.
 It would be interesting to establish a multiplicative Chow--K\"unneth decomposition for those varieties as well; I hope to return to this in the future.

  \vskip0.6cm

\begin{convention} In this note, the word {\sl variety\/} will refer to a reduced irreducible scheme of finite type over $\C$. For a smooth variety $X$, we will denote by $A^j(X)$ the Chow group of codimension $j$ cycles on $X$ 
with $\QQ$--coefficients.

The notation 
$A^j_{hom}(X)$ will be used to indicate the subgroups of 
homologically trivial cycles.

For a morphism between smooth varieties $f\colon X\to Y$, we will write $\Gamma_f\in A^\ast(X\times Y)$ for the graph of $f$.
The contravariant category of Chow motives (i.e., pure motives with respect to rational equivalence as in \cite{Sc}, \cite{MNP}) will be denoted $\MM_{\rm rat}$. 

We will write $H^\ast(X):=H^\ast(X,\QQ)$ for singular cohomology with $\QQ$--coefficients.
\end{convention}
  
 \vskip0.6cm

  \section{Multiplicative Chow--K\"unneth decomposition}

	\begin{definition}[Murre \cite{Mur}]\label{ck} Let $X$ be a smooth projective
		variety of dimension $n$. We say that $X$ has a 
		{\em CK  decomposition\/} if there exists a decomposition of the
		diagonal
		\[ \Delta_X= \pi^0_X+ \pi^1_X+\cdots +\pi^{2n}_X\ \ \ \hbox{in}\
		A^n(X\times X)\ ,\]
		such that the $\pi^i_X$ are mutually orthogonal idempotents and
		$(\pi^i_X)_\ast H^\ast(X)= H^i(X)$.
		Given a CK decomposition for $X$, we set 
		$$A^i(X)_{(j)} := (\pi_X^{2i-j})_\ast A^i(X).$$
		The CK decomposition is said to be {\em self-dual\/} if
		\[ \pi^i_X = {}^t \pi^{2n-i}_X\ \ \ \hbox{in}\ A^n(X\times X)\ \ \ \forall
		i\ .\]
		(Here ${}^t \pi$ denotes the transpose of a cycle $\pi$.)
		
		  (NB: ``CK decomposition'' is short-hand for ``Chow--K\"unneth
		decomposition''.)
	\end{definition}
	
	\begin{remark} \label{R:Murre} The existence of a Chow--K\"unneth decomposition
		for any smooth projective variety is part of Murre's conjectures \cite{Mur},
		\cite{MNP}. 
		It is expected that for any $X$ with a CK
		decomposition, one has
		\begin{equation*}\label{hope} A^i(X)_{(j)}\stackrel{??}{=}0\ \ \ \hbox{for}\
		j<0\ ,\ \ \ A^i(X)_{(0)}\cap A^i_{num}(X)\stackrel{??}{=}0.
		\end{equation*}
		These are Murre's conjectures B and D, respectively.
	\end{remark}

	\begin{definition}[Definition 8.1 in \cite{SV}]\label{mck} Let $X$ be a
		smooth
		projective variety of dimension $n$. Let $\Delta_X^{sm}\in A^{2n}(X\times
		X\times X)$ be the class of the small diagonal
		\[ \Delta_X^{sm}:=\bigl\{ (x,x,x) : x\in X\bigr\}\ \subset\ X\times
		X\times X\ .\]
		A CK decomposition $\{\pi^i_X\}$ of $X$ is {\em multiplicative\/}
		if it satisfies
		\[ \pi^k_X\circ \Delta_X^{sm}\circ (\pi^i_X\otimes \pi^j_X)=0\ \ \ \hbox{in}\
		A^{2n}(X\times X\times X)\ \ \ \hbox{for\ all\ }i+j\not=k\ .\]
		In that case,
		\[ A^i(X)_{(j)}:= (\pi_X^{2i-j})_\ast A^i(X)\]
		defines a bigraded ring structure on the Chow ring\,; that is, the
		intersection product has the property that 
		\[  \ima \Bigl(A^i(X)_{(j)}\otimes A^{i^\prime}(X)_{(j^\prime)}
		\xrightarrow{\cdot} A^{i+i^\prime}(X)\Bigr)\ \subseteq\ 
		A^{i+i^\prime}(X)_{(j+j^\prime)}\ .\]
		
		(For brevity, we will write {\em MCK decomposition\/} for ``multiplicative Chow--K\"unneth decomposition''.)
	\end{definition}
	
	\begin{remark}
	The property of having an MCK decomposition is
	severely restrictive, and is closely related to Beauville's ``(weak) splitting
	property'' \cite{Beau3}. For more ample discussion, and examples of varieties
	admitting a MCK decomposition, we refer to
	\cite[Chapter 8]{SV}, as well as \cite{V6}, \cite{SV2},
	\cite{FTV}, \cite{LV}.
	\end{remark}

There are the following useful general results:	

\begin{proposition}[Shen--Vial \cite{SV}]\label{product} Let $M,N$ be smooth projective varieties that have an MCK decomposition. Then the product $M\times N$ has an MCK decomposition.
\end{proposition}

\begin{proof} This is \cite[Theorem 8.6]{SV}, which shows more precisely that the {\em product CK decomposition\/}
  \[ \pi^i_{M\times N}:= \sum_{k+\ell=i} \pi^k_M\times \pi^\ell_N\ \ \ \in A^{\dim M+\dim N}\bigl((M\times N)\times (M\times N)\bigr) \]
  is multiplicative.
\end{proof}

\begin{proposition}[Shen--Vial \cite{SV2}]\label{blowup} Let $M$ be a smooth projective variety, and let $f\colon\wt{M}\to M$ be the blow--up with center a smooth closed subvariety
$N\subset M$. Assume that
\begin{enumerate}

\item $M$ and $N$ have a self--dual MCK decomposition;

\item the Chern classes of the normal bundle $\NNN_{N/M}$ are in $A^\ast_{(0)}(N)$;

\item the graph of the inclusion morphism $N\to M$ is in $A^\ast_{(0)}(N\times M)$;

\item the Chern classes $c_j(T_M)$ are in $A^\ast_{(0)}(M)$.

\end{enumerate}
Then $\wt{M}$ has a self--dual MCK decomposition, the Chern classes $c_j(T_{\wt{M}})$ are in $A^\ast_{(0)}(\wt{M})$, and the graph $\Gamma_f$ is in $A^\ast_{(0)}(  \wt{M}\times M)$.
\end{proposition}

\begin{proof} This is \cite[Proposition 2.4]{SV2}.
\end{proof}

 \section{Main result}
 
 \begin{theorem}\label{main} Let $X$ be a K\"uchle fourfold of type d3. Then $X$ has a self--dual MCK decomposition. Moreover, the Chern classes $c_j(T_X)$ are in 
 $A^\ast_{(0)}(X)$.
 \end{theorem} 
 
 \begin{proof} The argument relies on the following alternative description of $X$ (this is \cite[Corollary 3.5]{Kuz}):
 
 \begin{theorem}[Kuznetsov \cite{Kuz}]\label{alter} Let $X$ be a K\"uchle fourfold of type d3. Then $X$ is isomorphic to the blow--up of $(\PP^1)^4$ in a $K3$ surface $S$ obtained as smooth intersection of two divisors of multidegree $(1,1,1,1)$.
 \end{theorem}
 
 (The argument of  \cite[Corollary 3.5]{Kuz} also shows that conversely, any blow--up of  $(\PP^1)^4$ in a $K3$ surface $S$ of this type is a K\"uchle fourfold of type d3. We will not need this.)
  
 We wish to apply the general result proposition \ref{blowup} to $M=(\PP^1)^4$ and $N=S$. All we need to do is to check that the assumptions of proposition \ref{blowup} are met with.
 Assumption (1) is OK, since varieties with trivial Chow groups and $K3$ surfaces have a self--dual MCK decomposition \cite[Example 8.17]{SV}. Assumption (4) is trivially satisfied, since $A^\ast_{hom}((\PP^1)^4)=0$ and so $A^\ast_{}((\PP^1)^4)=A^\ast_{(0)}((\PP^1)^4)$.
 
 To check assumptions (2) and (3), we consider things family--wise. That is, we write
   \[ \bar{B}:= \PP H^0\bigl( (\PP^1)^4, \OO(1,1,1,1)^{\oplus 2}\bigr) \]
   and we consider the universal complete intersection
   \[\bar{\Ss}\ \to\ \bar{B}\ .\] 
   We write $B\subset\bar{B}$ for the Zariski open parametrizing smooth dimensionally transversal intersection, and $\Ss\to B$ for the base change (so the fibres $S_b$ of $\Ss\to B$ are exactly the $K3$ surfaces considered in theorem \ref{alter}).
   We make the following claim:
   
   \begin{claim}\label{gfc} Let $\Gamma\in A^i(\Ss)$ be such that 
     \[ \Gamma\vert_{S_b}=0\ \ \ \hbox{in}\ H^{2i}(S_b)\ \ \ \forall b\in B\ .\]
     Then also
      \[ \Gamma\vert_{S_b}=0\ \ \ \hbox{in}\ A^i(S_b)\ \ \ \forall b\in B\ .\]
     \end{claim}
     
     We argue that the claim implies that assumptions (2) and (3) are met with (and so proposition \ref{blowup} can be applied to prove theorem \ref{main}).
     Indeed, let $p_j\colon \Ss\times_B \Ss\to \Ss$, $j=1,2$, denote the two projections.
     We observe that
     \[ \pi^0_\Ss:={1\over 24} (p_1)^\ast  c_2(T_{\Ss/B})\ ,\ \ \pi^4_\Ss:={1\over 24} (p_2)^\ast  c_2(T_{\Ss/B})\ ,\ \ \pi^2_\Ss:= \Delta_\Ss-\pi^0_\Ss-\pi^4_\Ss\ \ \ \in A^4(\Ss\times_B \Ss) \]
     defines a ``relative MCK decomposition'', in the sense that for any $b\in B$, the restriction $\pi^i_\Ss\vert_{S_b\times S_b}$ defines an MCK decomposition (which we denote
     $\pi^i_{S_b}$) for $S_b$.
     
     To check that assumption (2) is satisfied, we need to check that for any $b\in B$ there is vanishing
     \begin{equation}\label{need} (\pi^2_{S_b})_\ast c_2(\NNN_{S_b/(\PP^1)^4})\stackrel{??}{=}0\ \ \ \hbox{in}\ A^2(S_b)\ .\end{equation}
     But we can write
      \[  (\pi^2_{S_b})_\ast c_2(\NNN_{S_b/(\PP^1)^4})  =  \Bigl(  (\pi^2_\Ss)_\ast c_2(\NNN_{\Ss/((\PP^1)^4\times B)})\Bigr)\vert_{S_b}    \ \ \ \hbox{in}\ A^2(S_b)\ \]
      (for the formalism of relative correspondences, cf. \cite[Chapter 8]{MNP}),
      and besides we know that $  (\pi^2_{S_b})_\ast c_2(\NNN_{S_b/(\PP^1)^4})$ is homologically trivial ($\pi^2_{S_b}$ acts as zero on $H^4(S_b)$). Thus, the claim implies the necessary vanishing (\ref{need}). 
      
     Assumption (3) is checked similarly. Let $\iota_b\colon S_b\to (\PP^1)^4$ and $\iota\colon \Ss\to (\PP^1)^4\times B$ denote the inclusion morphisms.
     To check assumption (3), we need to convince ourselves of the vanishing    
      \begin{equation}\label{need2} (\pi^\ell_{S_b\times(\PP^1)^4})_\ast (\Gamma_{\iota_b})\stackrel{??}{=}0\ \ \ \hbox{in}\ A^4(S_b\times (\PP^1)^4)\ \ \ \forall \ell\not= 8\ ,\ \forall b\in B\ .\end{equation}
      
      The fact that $\ell\not=8$ implies that $ (\pi^\ell_{S_b\times(\PP^1)^4})_\ast (\Gamma_{\iota_b})$ is homologically trivial. Furthermore, we can write the cycle we are interested in as the restriction of a universal cycle:
      \[  (\pi^\ell_{S_b\times(\PP^1)^4})_\ast (\Gamma_{\iota_b}) =   \Bigl( (\sum_{j+k=\ell} \pi^j_\Ss\times \pi^k_{(\PP^1)^4}) _\ast (\Gamma_\iota)\Bigr)\vert_{S_b\times (\PP^1)^4}
      \ \ \ \hbox{in}\ A^4(S_b\times (\PP^1)^4)\ .\]
      For any $b\in B$, there is a commutative diagram
      \[  \begin{array}[c]{ccc}
              A^4(\Ss\times (\PP^1)^4) &\to&    A^4(S_b\times (\PP^1)^4)      \\
              &&\\
            \ \ \   \downarrow{\cong}&&\ \ \  \downarrow{\cong}\\
            &&\\
             \bigoplus A^\ast(\Ss) &\to&  \bigoplus A^\ast(S_b)\\
             \end{array} \]
             where horizontal arrows are restriction to a fibre, and
             where vertical arrows are isomorphisms by repeated application of the projective bundle formula for Chow groups.
            Claim \ref{gfc} applied to the lower horizontal arrow shows the vanishing (\ref{need2}), and so assumption (3) holds.
            
            It is left to prove the claim. Since $A^i_{hom}(S_b)=0$ for $i\le 1$, the only non--trivial case is $i=2$. 
            Given $\Gamma\in A^2(\Ss)$ as in the claim, let $\bar{\Gamma}\in A^2(\bar{\Ss})$ be a cycle restricting to $\Gamma$.
            We consider the two projections
             \[ \begin{array}[c]{ccc}
       \bar{\Ss}&\xrightarrow{\pi}& (\PP^1)^4  \\
       \ \ \ \ \downarrow{\scriptstyle \phi}&&\\
         \ \  \bar{B}\ &&\\
          \end{array}\]  
      Since any point of $(\PP^1)^4$ imposes exactly one condition on $\bar{B}$, the morphism $\pi$ has the structure of a projective bundle. As such, any 
      $\bar{\Gamma}\in A^{2}(\bar{\Ss})$ can be written
          \[ \bar{\Gamma}=    \sum_{\ell=0}^2 \pi^\ast( a_\ell)  \cdot \xi^\ell  \ \ \ \hbox{in}\ A^{2}(\bar{\Ss})\ ,\]
                where $a_\ell\in A^{2-\ell}( (\PP^1)^4)$ and $\xi\in A^1(\bar{\Ss})$ is the relative hyperplane class.
                  
        Let $h:=c_1(\OO_{\bar{B}}(1))\in A^1(\bar{B})$. There is a relation
        \[  \phi^\ast(h)=\alpha \xi +  \pi^\ast(h_1)\ \ \ \hbox{in}\ A^1(\bar{\Ss})\ ,\]
        where $\alpha\in\QQ$ and $h_1\in A^1((\PP^1)^4)$. As in \cite[Proof of Lemma 1.1]{PSY}, one checks that $\alpha\not=0$ (if $\alpha$ were $0$, we would have
        $\phi^\ast(h^{\dim \bar{B}})=\pi^\ast(h_1^{\dim \bar{B}})$, which is absurd since $\dim \bar{B}>4$ and so the right--hand side is $0$).
        Hence, there is a relation
        \[ \xi = {1\over \alpha} \bigl(\phi^\ast(h)-\pi^\ast(h_1)\bigr)\ \ \ \hbox{in}\ A^1(\bar{\Ss})\ .\]        
      For any $b\in B$, the restriction of $\phi^\ast(h)$ to the fibre $S_b$ vanishes, and so it follows that
      \[ \bar{\Gamma}\vert_{S_b} = a_0^\prime\vert_{S_b}\ \ \ \hbox{in}\ A^{2}(S_b)\ \]
     for some $a_0^\prime\in A^2( (\PP^1)^4)$. But
      \[ A^2( (\PP^1)^4)= \bigoplus_{i+j+k+\ell=2} A^i(\PP^1)\otimes A^j(\PP^1)\otimes A^k(\PP^1)\otimes A^\ell(\PP^1) \]
      is generated by intersections of divisors, and so Beauville--Voisin's result \cite{BV} implies that
      \[ \bar{\Gamma}\vert_{S_b} = a_0^\prime\vert_{S_b} \ \ \in\ \QQ {\mathfrak{o}}_{S_b}\ \ \ \subset\ A^{2}(S_b)\ ,\]
      where ${\mathfrak{o}}_{S_b}\in A^2(S_b)$ is the distinguished $0$--cycle of \cite{BV}. This proves the claim.
 \end{proof}

 \section{A consequence} 
  
 \begin{corollary}\label{cor} Let $X$ be a K\"uchle fourfold of type d3, and let $m\in\NN$. Let $R^\ast(X^m)\subset A^\ast(X^m)$ be the $\QQ$--subalgebra
   \[ R^\ast(X^m):=  < (p_i)^\ast A^1(X), (p_i)^\ast A^2(X), (p_{ij})^\ast(\Delta_X), (p_i)^\ast c_3(T_X)>\ \ \ \subset\ A^\ast(X^m)\ .\]
   (Here $p_i\colon X^m\to X$ and $p_{ij}\colon X^m\to X^2$ denote projection to the $i$th factor, resp. to the $i$th and $j$th factor.)
   
  The cycle class map induces injections
   \[ R^j(X^m)\ \hookrightarrow\ H^{2j}(X^m)\]
   in the following cases:
   
   \begin{enumerate}
   
   \item $m=1$ and $j$ arbitrary;
   
   \item $m=2$ and $j\ge 5$;
   
   \item $m=3$ and $j\ge 9$.
   \end{enumerate}
       \end{corollary}
 
 \begin{proof} Theorem \ref{main}, in combination with proposition \ref{product}, ensures that $X^m$ has an MCK decomposition, and so $A^\ast(X^m)$ has the structure of a bigraded ring under the intersection product. The corollary is now implied by the combination of the two following claims:

\begin{claim}\label{c1} There is inclusion
  \[ R^\ast(X^m)\ \ \subset\ A^\ast_{(0)}(X^m)\ .\]
  \end{claim}
  
 \begin{claim}\label{c2} The cycle class map induces injections
   \[ A^j_{(0)}(X^m)\ \hookrightarrow\ H^{2j}(X^m)\ \] 
   provided $m=1$, or $m=2$ and $j\ge 5$, or $m=3$ and $j\ge 9$.
\end{claim}

To prove claim \ref{c1}, we note that $A^k_{hom}(X)=0$ for $k\le 2$, which readily implies the equality $A^k(X)=A^k_{(0)}(X)$ for $k\le 2$. The fact that $c_3(T_X)$ is in $A^3_{(0)}(X)$ is part of theorem \ref{main}. The fact that $\Delta_X\in A^4_{(0)}(X\times X)$ is a general fact for any $X$ with a self--dual MCK decomposition \cite[Lemma 1.4]{SV2}. Since the projections $p_i$ and $p_{ij}$ are pure of grade $0$ \cite[Corollary 1.6]{SV2}, and $A^\ast_{(0)}(X^m)$ is a ring under the intersection product, this proves claim \ref{c1}.

To prove claim \ref{c2}, we observe that Manin's blow--up formula \cite[Theorem 2.8]{Sc} gives an isomorphism of motives
  \[ h(X)\cong h(S)(1)\oplus {\mathds{1}} \oplus  {\mathds{1}}(1)^{\oplus 4} \oplus {\mathds{1}}(2)^{\oplus 6}   \oplus  {\mathds{1}}(3)^{\oplus 4} \oplus  {\mathds{1}}(4)\ \ \ \hbox{in}\ \MM_{\rm rat}\ .\]
  Moreover, in view of proposition \ref{blowup} (cf. also \cite[Proposition 2.4]{SV2}), the correspondence inducing this isomorphism is of pure grade $0$.
  
 In particular, for any $m\in\NN$ we have isomorphisms of Chow groups
  \[ A^j(X^m)\cong A^{j-m}(S^m)\oplus \bigoplus_{k=0}^4  A^{j-m+1-k}(S^{m-1})^{b_k} \oplus \bigoplus A^\ast(S^{m-2})\oplus \bigoplus_{\ell\ge 3} A^\ast(S^{m-\ell})  \ ,  \] 
  and this isomorphism respects the $A^\ast_{(0)}()$ parts. Claim \ref{c2} now follows from the fact that for any surface $S$ with an MCK decomposition, and any $m\in\NN$, the cycle class map induces injections
  \[ A^i_{(0)}(S^m)\ \hookrightarrow\ H^{2i}(S^m)\ \ \ \forall i\ge 2m-1\ \]
  (this is noted in \cite[Introduction]{V6}, cf. also \cite[Proof of Lemma 2.20]{acs}).

\end{proof}

%
%
%
%
%
            
\vskip1cm
\begin{nonumberingt} Thanks to my colleagues Kai--kun and Len--kun at the Schiltigheim Centre for Pure Mathematics.

\end{nonumberingt}


\vskip1cm
\begin{thebibliography}{dlPG99}


\bibitem{Beau3} A. Beauville, On the splitting of the Bloch--Beilinson
		filtration, in: Algebraic cycles and motives (J. Nagel and C. Peters, editors), London Math. Soc. Lecture Notes 344, Cambridge University Press 2007,

\bibitem{BV} A. Beauville and C. Voisin, On the Chow ring of a $K3$ surface, J. Alg. Geom. 13 (2004), 417--426,




\bibitem{FM} E. Fatighenti and G. Mongardi, A note on a Griffiths--type ring for complete intersections in Grassmannians, arXiv:1801.09586v1,



\bibitem{FTV} L. Fu, Zh. Tian and Ch. Vial, Motivic hyperK\"ahler resolution conjecture: I. Generalized Kummer varieties, arXiv:1608.04968,

\bibitem{Ku} O. K\"uchle, On Fano $4$--folds of index $1$ and homogeneous vector bundles over Grassmannians, Math. Zeitschrift 218 (1995), 563---575, 

\bibitem{Kuz} A. Kuznetsov, On K\"uchle varieties with Picard number greater than $1$, Izvestiya RAN: Ser. Mat. 79:4 (2015), 57---70 (in Russian); translation in 
Izvestiya: Mathematics 79:4 (2015), 698---709,

\bibitem{acs} R. Laterveer, Algebraic cycles on some special hyperk\"ahler varieties, Rendiconti di Matematica e delle sue applicazioni 38 no. 2 (2017), 243---276, 

\bibitem{LV} R. Laterveer and Ch. Vial, On the Chow ring of Cynk--Hulek Calabi--Yau varieties and Schreieder
		varieties, arXiv:1712.03070,
		
\bibitem{Mur} J. Murre, On a conjectural filtration on the Chow groups of an
		algebraic variety, parts I and II, Indag. Math. 4 (1993), 177--201,

\bibitem{MNP} J. Murre, J. Nagel and C. Peters, Lectures on the theory of pure motives, Amer. Math. Soc. University Lecture Series 61, Providence 2013,


\bibitem{PSY} N. Pavic, J. Shen and Q. Yin, On O'Grady's generalized Franchetta conjecture, Int. Math. Res. Notices (2016), 1---13,

\bibitem{Sc} T. Scholl, Classical motives, in: Motives (U. Jannsen et alii, eds.), Proceedings of Symposia in Pure Mathematics Vol. 55 (1994), Part 1,

\bibitem{SV} M. Shen and Ch. Vial, The Fourier transform for certain hyperK\"ahler fourfolds, Memoirs of the AMS 240 (2016), 
		
\bibitem{SV2} M. Shen and Ch. Vial, On the motive of the Hilbert cube $X^{[3]}$, Forum Math. Sigma 4 (2016),
		
\bibitem{V6} Ch. Vial, On the motive of some hyperk\"ahler varieties, J. f\"ur Reine u. Angew. Math. 725 (2017), 235--247,



\end{thebibliography}
\end{document}